\documentclass[]{interact}

\usepackage{epstopdf}
\usepackage[caption=false]{subfig}

\usepackage[numbers,sort&compress]{natbib}
\bibpunct[, ]{[}{]}{,}{n}{,}{,}
\renewcommand\bibfont{\fontsize{10}{12}\selectfont}

\theoremstyle{plain}
\newtheorem{theorem}{Theorem}[section]
\newtheorem{lemma}[theorem]{Lemma}

\newtheorem{proposition}[theorem]{Proposition}
\newtheorem{assumption}[theorem]{Assumption}

\theoremstyle{definition}
\newtheorem{definition}[theorem]{Definition}

\theoremstyle{remark}

\usepackage[utf8]{inputenc}
\usepackage{amsmath}
\usepackage{amssymb}
\usepackage{amsthm}
\usepackage{algorithm}
\usepackage{algorithmicx}
\usepackage{algpseudocode}
\usepackage{multirow}
\usepackage{longtable}
\usepackage{rotating, tabularx}
\usepackage[numbers,sort&compress]{natbib}
\bibpunct[, ]{[}{]}{,}{n}{,}{,}
\renewcommand\bibfont{\fontsize{10}{12}\selectfont}

\begin{document}


\title{Incremental Quasi-Newton Algorithms for Solving Nonconvex, Nonsmooth, Finite-Sum Optimization Problems}

\author{
\name{Gulcin Dinc Yalcin\textsuperscript{a}\thanks{CONTACT G. Dinc Yalcin. Email: gdinc@eskisehir.edu.tr ORCID ID: 0000-0001-7696-7507} and Frank E. Curtis\textsuperscript{b}\thanks{Frank E. Curtis ORCID ID: 0000-0001-7214-9187}}
\affil{\textsuperscript{a}Department of Industrial Engineering, Eskisehir Technical University, Eskisehir, 26555, Türkiye; \textsuperscript{a,b}Department of Industrial and Systems Engineering, Lehigh University, Bethlehem, PA 18015, USA}
}

\maketitle
\begin{abstract}
Algorithms for solving nonconvex, nonsmooth, finite-sum optimization problems are proposed and tested.  In particular, the algorithms are proposed and tested in the context of an optimization problem formulation arising in semi-supervised machine learning.  The common feature of all algorithms is that they employ an incremental quasi-Newton (IQN) strategy, specifically an incremental BFGS (IBFGS) strategy.  One applies an IBFGS strategy to the problem directly, whereas the others apply an IBFGS strategy to a difference-of-convex reformulation, smoothed approximation, or (strongly) convex local approximation.  Experiments show that all IBFGS approaches fare well in practice, and all outperform a state-of-the-art bundle method.
\end{abstract}

\section*{2010 MATHEMATICS SUBJECT CLASSIFICATIONS}
49M37, 65K05, 90C26, 90C30, 90C53

\section{Introduction}\label{sec.introduction}

In this paper, we investigate the performance of algorithms for solving nonconvex, nonsmooth, finite-sum minimization problems that arise in areas such as statistical/machine learning.  In particular, given a set of $m \in \mathbb{N}$ functions that might be nonconvex and/or nonsmooth, say, $f_i : {\cal R} \to \mathbb{R}$ for all $i \in \{1,\dots,m\}$ (where $\cal R$ is some finite-dimensional real vector space), we are interested in algorithms designed to minimize the potentially nonconvex and/or nonsmooth objective $f : {\cal R} \to \mathbb{R}$ as in
\begin{equation}\label{problem}
  \min_{\omega \in \cal R}\ f(\omega),\ \ \text{where}\ \ f(\omega) := \tfrac{1}{m} \sum_{i=1}^m f_i(\omega).
\end{equation}
Specifically, to demonstrate the design and use of a suite of related algorithms, we focus on a particular binary classification problem arising in semi-supervised machine learning, which can be formulated as follows.  Suppose one is given a set of $p \in \mathbb{N}$ \emph{labeled} training samples $\{(x_i,y_i)\}_{i=1}^p$ with $x_i \in \mathbb{R}^n$ and $y_i \in \{-1,1\}$ for all $i \in \{1,\dots,p\}$ as well as a set of $q \in \mathbb{N}$ \emph{unlabeled} training samples $\{x_i\}_{p+1}^{p+q}$ with $x_i \in \mathbb{R}^n$ for all $i \in \{p+1,\dots,p+q\}$.  The \emph{transductive support vector machine} (TSVM) problem~\cite{bennett1998}, similar to the classical support vector machine (SVM) problem, aims to find a hyperplane that separates points in the positively labeled class from those in the negatively labeled class (labeled by $1$ and $-1$, respectively).  An unconstrained optimization formulation of the TSVM problem (see, e.g., \cite{AstFud2007,joachims1999}), can be written in terms of optimization variables $w \in \mathbb{R}^n$ and $b \in \mathbb{R}$ as the problem to minimize $F : \mathbb{R}^n \times \mathbb{R} \to \mathbb{R}$, where
\begin{equation}\label{semisupervised}
  F(w,b) = \tfrac{1}{2} \|w\|_2^2 + C_1 \sum_{i=1}^p \max\{0, 1 - y_i(w^Tx_i +b)\} + C_2 \sum_{i=p+1}^{p+q} \max\{0, 1 - |w^Tx_i+b|\},
\end{equation}
with $C_1 \in (0,\infty)$ and $C_2 \in (0,\infty)$ being user-defined parameters.  In this definition of~$F$, the first term is strongly convex, the second term is convex and nonsmooth, and the third term is nonconvex and nonsmooth; hence, $F$ is nonconvex and nonsmooth.

The algorithms that we propose and investigate are all built upon an incremental quasi-Newton (IQN) framework, specifically the incremental version of the Broyden-Fletcher-Goldfarb-Shanno (BFGS) \cite{broyden1970,fletcher1970,goldfarb1970,shanno1970} strategy known as IBFGS \cite{Mokhtari2018}.  As discussed in Section~\ref{Literature}, quasi-Newton methods (such as BFGS) are designed to mimic the behavior of Newton's method for optimization through only the use of first-order derivative information, whereas incremental quasi-Newton is designed to further exploit the finite-sum structure of an objective function.  Since our problem of interest, namely, \eqref{semisupervised}, is nonconvex and nonsmooth, the convergence guaranteees for the original IBFGS algorithm~\cite{Mokhtari2018} do not hold in our setting, although we still investigate the performance of the approach applied directly to the problem to see if it can still be efficient and effective in practice.  We also propose and investigate the performance of IBFGS strategies applied to solve a few reformulations and approximations of problem~\eqref{semisupervised}, namely, one that uses a difference-of-convex reformulation, one that uses a smoothed approximation, and others that use (strongly) convex local approximations.

\subsection{Contributions}

Our main goal is to answer the question: Can IQN (such as IBFGS) strategies be effective and efficient when solving the nonconvex and nonsmooth problem \eqref{semisupervised}, which would provide evidence that they should be considered viable options more generally when solving problems of the form \eqref{problem}?  Our answer is in the affirmative, as we show in numerical experiments that a straightforward application of an IBFGS method, as well as application of all of our proposed reformulations and approximations, are successful in our experiments, and all outperform a state-of-the-art bundle method.  Convergence guarantees for IBFGS in our nonconvex and nonsmooth setting are elusive, as they are elusive for numerous other algorithms for solving nonconvex and nonsmooth problems that are applied regularly in practice to great effect, and as they are elusive even for non-incremental quasi-Newton strategies in general in convex and nonsmooth settings.  Nevertheless, we show that our proposed strategies, which involve either applying an IBFGS strategy directly or to a mollified reformulation/approximation, are effective when solving challenging problems motivated by an important real-world application.

\subsection{Literature Review}\label{Literature}

Steepest descent methods (also known as gradient methods) for optimization require only first-order derivative information and can achieve a local linear rate of convergence in favorable settings.  Newton methods, on the other hand, use second-order derivative information to obtain a faster rate of local convergence---namely, at least quadratic in favorable cases---but this comes at the cost of additional computation to solve a large-scale linear system in each iteration.  Quasi-Newton methods offer a balance between per-iteration computational cost and convergence rate by approximating second-order derivatives using only first-order derivative information.  The BFGS method is one of the most popular quasi-Newton methods in practice \cite{Nocedal2006}.  Like other quasi-Newton strategies, BFGS uses differences of gradients at consecutive iterates to update iteratively an approximation of the (inverse) Hessian of the objective.  Convergence guarantees for the BFGS method have been proved for convex smooth (see e.g. \cite{byrd1987,dennis1974}) and nonconvex smooth (see e.g. \cite{li2001mod,li2001}) objectives.  BFGS methods have also been studied for the minimization of nonsmooth functions (see, e.g., \cite{curtis2015,curtis2020}), although convergence guarantees for nonsmooth minimization are very limited to certain convex problems (see \cite{lewis2013} for some example convergence guarantees for specific functions).

Steepest descent, Newton, and quasi-Newton methods have also been designed specifically for the case of minimizing objectives defined by finite sums, such as in \eqref{problem}.  For example, incremental gradient methods use the gradient of only a single component function in each iteration, cycling through the component functions using a prescribed order or strategy.  These methods can significantly outperform non-incremental methods, although in theory they may only achieve a sublinear rate of convergence~\cite{bertsekas2011}. The convergence properties of such methods for convex minimization have been investigated in various articles \cite{blatt2007,gurbuzbalaban2017,solodov1998,wai2020}.  For convex nonsmooth minimization, incremental subgradient methods have been analyzed in, e.g., \cite{iiduka2016,hu2019,kiwiel2004,nedic2001,ram2009}.

Beyond incremental (sub)gradient methods, \cite{gurbuzbalaban2015} presents an incremental Newton method for cases when the component functions are assumed to be strongly convex.  In \cite{gurbuzbalaban2019}, the convergence rate of incremental gradient and incremental Newton methods are studied under constant and diminishing step size rules.  For the same reasons as in the non-incremental setting, IQN methods have also been proposed for the incremental setting.  After all, applying a method such as BFGS when solving problem \eqref{problem} requires ${\cal O}(mn+n^2)$ operations per iteration, which can be large when $m$ and/or $n$ is large.  The IBFGS strategy proposed in \cite{Mokhtari2018} has a cost of only ${\cal O}(n^2)$ per iteration, and has convergence guarantees when the objective is convex and smooth.  It is this algorithm that we use as a basis for the methods proposed in this paper.

Various other problems of the form \eqref{problem}, some related to \eqref{semisupervised}, have been studied in the literature, as well as algorithms for solving them; see, e.g., \cite{joachims1999,chapelle2005,belkin2006,sindhwani2006,collobert2006,AstFud2007}.  We also refer the reader to review papers and book chapters, such as \cite{chapelle2008,van2020,zhou2014}.  This being said, to the best of our knowledge, there has not been prior work on the variants of the IBFGS strategy that are proposed and investigated in this paper.  We also mention in passing that stochastic (quasi-)Newton methods have been studied for solving machine learning problems as well \cite{byrd2011,byrd2016,qi2017,zhao2017}.

\subsection{Organization}

In Section~\ref{SSL}, we present reformulations and approximations of~\eqref{semisupervised} to which our proposed algorithms are to be applied.  In Section~\ref{IQN}, we provide background on quasi-Newton methods, BFGS, and IBFGS, then propose variants of IBFGS for solving the problems formulations that are presented in Section~\ref{SSL}.  Section~\ref{computational} provides the results of numerical experiments using our proposed algorithms and a state-of-the-art bundle method for the sake of comparison.  Concluding remarks are given in Section~\ref{conclusion}.

\section{Reformulations and Approximations}\label{SSL}

As mentioned, the semi-supervised machine learning objective function \eqref{semisupervised} involves three terms---one strongly convex, one convex and nonsmooth, and one nonconvex and nonsmooth---leading overall to a function that is nonconvex and nonsmooth.  Like other such functions that appear in the literature, however, this function has structure that can be exploited in the context of employing optimization algorithms.  In this section, we present a few potential reformulations and approximations of this function that could be useful in practice; indeed, our experiments confirm that they are useful.

\subsection{Difference-of-Convex (DC) Function Reformulation}\label{sec.DC}

If one is willing to accept the nonsmoothness of \eqref{semisupervised}, then the remaining obstacle is nonconvexity.  Fortunately, \eqref{semisupervised} has the property that it can be expressed as a special type of nonconvex function, namely, a difference-of-convex (DC) function, which is a type that is now well known to have properties that can be exploited in the context of optimization \cite{CuiPang2021}.  Formally, let us introduce the following definition.

\begin{definition} (see, e.g., \cite[Definition 2.1]{horst1999})
  \emph{A continuous function $f : \cal R \to \mathbb{R}$ is a difference-of-convex (DC) function if and only if there exist two convex functions, call them $g : \cal R \to \mathbb{R}$ and $h : \cal R \to \mathbb{R}$, such that $f = g - h$.}
\end{definition}

Since the first two terms in \eqref{semisupervised} are convex, to show that $F$ is a DC function, all that is needed is to express the third term as a DC function.  This can be done by recognizing that, for all $i \in \{p+1,\dots,p+q\}$, one can write
\begin{equation*}
  \begin{aligned}
    f_i(w,b) :=&\ \max\{0,1-|w^Tx_i + b|\} \\
    =&\ \max\{0,|w^T x_i+b|-1\} - (|w^T x_i+b|-1) =: g_i(w,b) - h_i(w,b).
  \end{aligned}
\end{equation*}
This decomposition is illustrated in Figure~\ref{fig:dc}.  Using this observation, a DC reformulation of \eqref{semisupervised} is given by $F^{DC} : \mathbb{R}^n \times \mathbb{R} \to \mathbb{R}$ defined by
\begin{multline}
  F^{DC}(w,b) = \tfrac{1}{2} \|w\|_2^2 + C_1 \sum_{i=1}^p \max\{0,1-y_i(w^T x_i + b)\} \\
  + C_2 \left( \sum_{i=p+1}^{p+q} \max\{|w^T x_i + b| - 1, 0\} - \sum_{i=p+1}^{p+q} (|w^T x_i+b|-1) \right).
\end{multline}

\begin{figure}[ht]
  \centering
  \includegraphics[width=0.33\textwidth]{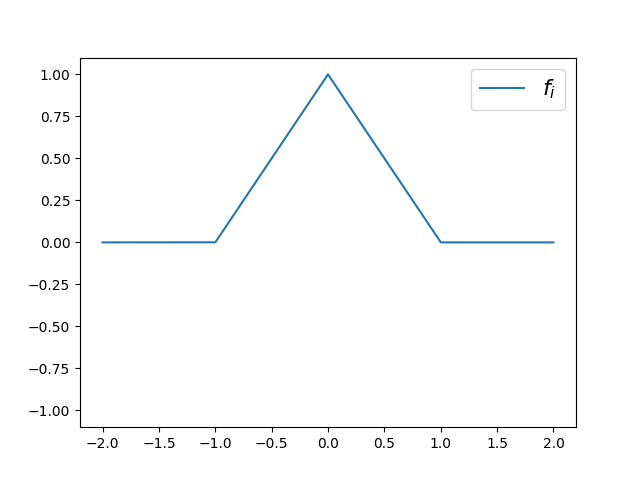}
  \includegraphics[width=0.33\textwidth]{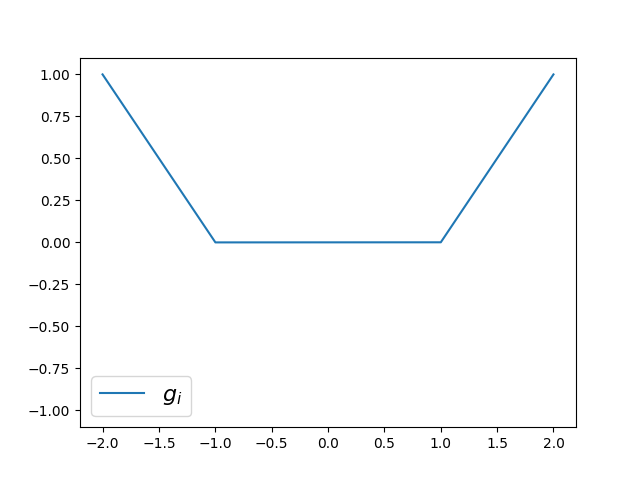}
  \includegraphics[width=0.33\textwidth]{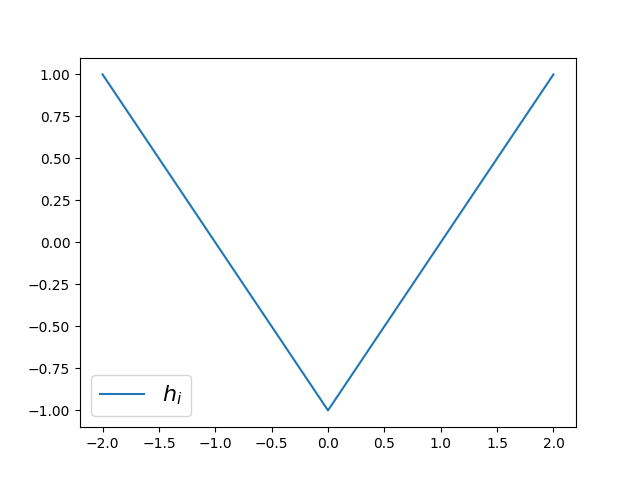}
  \caption{DC function formulation of $f_i(\cdot) = \max\{0,1-|\cdot|\}$ (on the left) as the difference between $g_i(\cdot) = \max\{0,|\cdot|-1\}$ (in the middle) and $h_i(\cdot) = |\cdot|-1$ (on the right).}
  \label{fig:dc}
\end{figure}

\subsection{Smooth Approximation}

With the idea of opening the door for the use of the numerous algorithms that have been designed for smooth optimization, one can consider avoiding the nonsmoothness of \eqref{semisupervised} by considering a smooth approximation.  In particular, one can consider attempting to minimize \eqref{semisupervised} by solving (approximately) a sequence of smooth approximations that tend to \eqref{semisupervised} in the limit.  For this purpose, let us introduce the following definition.

\begin{definition}(see, e.g., \cite[Definition 1]{chen})
  \emph{A function $f^S : {\cal R} \times (0,\infty) \to \mathbb{R}$ is a smoothing function of a continuous function $f : \cal R \to \mathbb{R}$ if and only if $f^S(\cdot,\mu)$ is continuously differentiable for all $\mu \in (0,\infty)$ and, for any $\omega \in \cal R$, one has
  \begin{equation*}
    \lim_{\bar\omega \to \omega, \mu \searrow 0} f^S(\bar\omega,\mu) = f(\omega).
  \end{equation*}}
\end{definition}

Fortunately, the nonsmoothness of \eqref{semisupervised} can be smoothed in a straightforward manner to obtain a smoothing function.  Following \cite{chen,ChenNiuYuan13}, for the univariate function $(t)_+ = \max\{0,t\}$, a smoothing function $\phi:\mathbb{R} \times \mathbb{R} \to \mathbb{R}$ is given by either
\begin{equation}\label{eq.smoothed_max}
  \phi(t,\mu) = \begin{cases} (t)_+ & \text{if} \quad |t|\geq \tfrac{\mu}{2} \\ \tfrac{t^2}{2\mu}+\tfrac{t}{2}+\tfrac{\mu}{8} & \text{if} \quad |t|<\tfrac{\mu}{2} \end{cases}\ \ \text{or}\ \ \phi(t,\mu) = \tfrac12 (t + \sqrt{t^2 + 4\mu^2}).
\end{equation}
Correspondingly, using the fact that $|t| = (t)_+ + (-t)_+$, a smoothing function $\psi:\mathbb{R} \times \mathbb{R} \to \mathbb{R}$ for the univariate absolute value function is given by either
\begin{equation}\label{eq.smoothed_abs}
  \psi(t,\mu) = \begin{cases} |t| & \text{if} \quad |t|\geq \tfrac{\mu}{2} \\ \tfrac{t^2}{\mu}+\tfrac{\mu}{4} & \text{if} \quad |t|<\tfrac{\mu}{2} \end{cases}\ \ \text{or}\ \ \psi(t,\mu) = \sqrt{t^2 + 4\mu^2}.
\end{equation}
Other choices are possible as well, each with different theoretical and/or practical advantages and disadvantages.  For example, the former options above have the advantage that they match $(t)_+$ and $|t|$, respectively, when $|t| \geq \tfrac\mu2$, which is not the case for the latter options.  However, the latter options have the advantage of being twice continuously differentiable, which is not the case for the former options.

One can now obtain a smooth approximation of \eqref{semisupervised} using smoothing functions.  First, for all $i \in \{1,\dots,p\}$, suppose a smoothing function $f_i^S$ of $f_i$ is given by $f_i^S(w,b,\mu) = \phi(t_i,\mu)$, where $t_i \equiv 1 - y_i(w^Tx_i + b)$.  Second, for all $i \in \{p+1,\dots,p+q\}$, suppose a smoothing function $f_i^S$ of $f_i$ is given by $f_i^S(w,b,\mu) = \phi(t_i,\mu)$, where $t_i \equiv 1 - \psi(w^Tx_i + b,\mu)$.  Then, a smoothing function $F^S$ for $F$ in \eqref{semisupervised} is given by
\begin{equation}\label{smooth}
  \begin{aligned}
    F^S(w,b,\mu) = \tfrac{1}{2} \|w\|_2^2 &+ C_1 \sum_{i=1}^p \phi(1 - y_i(w^T x_i + b),\mu) \\
    &+ C_2 \sum_{i=p+1}^{p+q} \phi(1-\psi(w^T x_i + b, \mu), \mu);
  \end{aligned}
\end{equation}
see Figure~\ref{fig:smoothing} for illustration.  Here, we are making use of the fact that a smooth composite of smoothing functions is a smoothing function; see, e.g., \cite[Theorem~1]{chen}.

\begin{figure}[ht]
  \centering
  \includegraphics[width=0.33\textwidth]{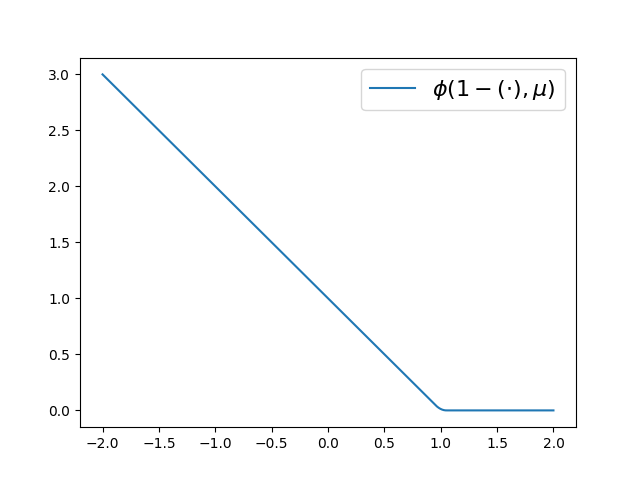}
  \includegraphics[width=0.33\textwidth]{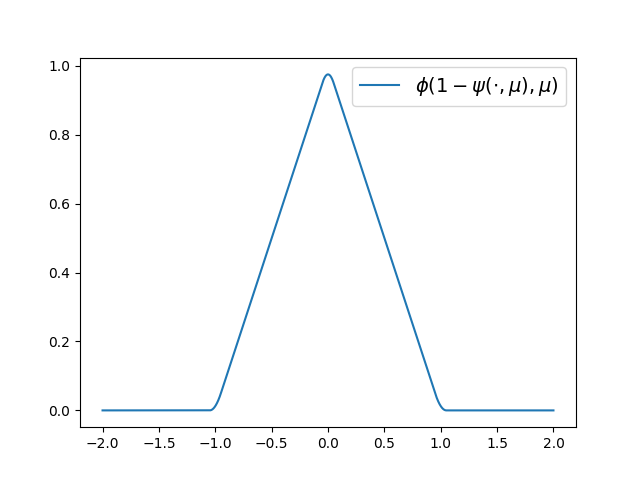}
  \caption{Illustration of $\phi(1 - (\cdot),\mu)$ (left) and $\phi(1 - \psi(\cdot,\mu),\mu)$ (right) with $\mu = 0.1$ using the former choices for the smoothing functions in \eqref{eq.smoothed_max} and \eqref{eq.smoothed_abs}.}
  \label{fig:smoothing}
\end{figure}

\subsection{Smooth Convex Local Approximation}\label{sec.convex}

The smoothing function \eqref{smooth} has the benefit of being continuously differentiable.  However, it is nonconvex, which presents certain challenges when one aims to employ optimization techniques such as a quasi-Newton strategy.  To address these challenges, one can consider---about some base point, call it $(\bar{w},\bar{b}) \in \mathbb{R}^n \times \mathbb{R}$---a convex ``local'' approximation of~$F^S$.  Thus, the question becomes: Can $F^S$ be ``convexified'' locally in a straightforward manner?  The answer is yes, since $F^S$ happens to be weakly convex.

\begin{definition}(see, e.g., \cite[Definition~4.1 and Proposition~4.3]{vial1983})
  \emph{A function $f : {\cal R} \to \mathbb{R}$ is $\rho$-convex if and only if there exists a convex function $g : {\cal R} \to \mathbb{R}$ such that $f(\omega) = g(\omega) + \rho\|x\|_2^2$.  A function that is $\rho$-convex for $\rho \in (0,\infty)$ is said to be strongly convex.  A function that is $\rho$-convex for $\rho \in (-\infty,0)$ is said to be weakly convex.}
\end{definition}

To see that $F^S$ is weakly convex, the following proposition is useful.

\begin{proposition} {\rm (see, e.g., \cite[Section~3.1]{davis2018})} \label{prop.weaklyconvex}
  If $c : {\cal R} \to \mathbb{R}$ is continuously differentiable with its gradient function being Lipschitz with constant $L_{\nabla c} \in (0,\infty)$, and if $h : \mathbb{R} \to \mathbb{R}$ is convex and Lipschitz with constant $L_h \in (0,\infty)$, then the composite function $h \circ c : {\cal R} \to \mathbb{R}$ is $\rho$-convex with $\rho = -L_hL_{\nabla c} \in (-\infty,0)$.
\end{proposition}

The first two terms for $F^S$ in \eqref{smooth} are convex.  Hence, to see that $F^S$ is weakly convex, it follows from Proposition~\ref{prop.weaklyconvex} that it suffices to notice that the final term is a composite function following the conditions of the proposition.  For example, let us consider the former options in \eqref{eq.smoothed_max} and \eqref{eq.smoothed_abs}.  One finds for any $\mu \in (0,\infty)$ that $\phi(\cdot,\mu)$ is convex and Lipschitz with constant 1.  In addition, one finds for any $\mu \in (0,\infty)$ and $i \in \{p+1,\dots,p+q\}$ that $1-\psi(w^Tx_i + b,\mu)$ is continuously differentiable with a gradient that is Lipschitz continuous with constant $2(\|x_i\|_2^2 + 1)/\mu$.  Hence, given any smoothing parameter $\mu \in (0,\infty)$ and base point $(\bar w, \bar b)$, a smooth and convex local approximation of $F^S$ is given by $F^C(\cdot, \cdot, \mu, \bar w, \bar b) : \mathbb{R}^n \times \mathbb{R} \to \mathbb{R}$ defined by
\begin{multline}\label{convex_formula}
  F^C(w,b,\mu,\bar w,\bar b) = \tfrac{1}{2} \|w\|_2^2 + C_1 \sum_{i=1}^p \phi_i(1 - y_i(w^T x_i + b),\mu) \\
  + C_2 \sum_{i=p+1}^{p+q} \left(\phi(1 - \psi(w^T x_i + b, \mu), \mu) + \rho_i \left\| \begin{bmatrix} w \\ b \end{bmatrix} - \begin{bmatrix} \bar w \\ \bar b \end{bmatrix} \right\|_2^2 \right),
\end{multline}
where for all $i \in \{p+1,\dots,p+q\}$ one has chosen $\rho_i \in (0,\infty)$ with $\rho_i \geq 2(\|x_i\|_2^2 + 1)/\mu$.

Notice that it has been critical in this discussion to consider the smooth approximation $F^S(\cdot,\mu)$ for some $\mu \in (0,\infty)$, rather than $F$ in \eqref{semisupervised} directly.  After all, the function~$F$ in~\eqref{semisupervised} is not weakly convex.  Therefore, it should not be surprising that the lower bounds for $\{\rho_i\}_{i=p+1}^{p+q}$ tend to infinity as $\mu \searrow 0$, i.e., as $F^S$ tends to $F$.

\subsection{Smooth Strongly Convex Local Approximation}\label{sec.strongly_convex}

The algorithms that we propose and test in the following sections are built on the fact that the objective functions under consideration consist of a finite sum of functions.  For example, at face value, such an algorithm could consider $\tfrac12 \|w\|_2^2$ as one function and the remaining $p+q$ terms as separate functions, meaning that one can consider there being $p+q+1$ separate terms overall.  However, one disadvantage of this breakdown is that not all of these terms would be strongly convex.  This is a disadvantage since, as we shall mention in the next section, quasi-Newton methods can avoid the need to \emph{damp} or \emph{skip} updates when a function is strongly convex.  Hence, it is worthwhile also to explore an alternative to \eqref{convex_formula} in which all terms are strongly convex.

One way to derive such an alternative is to simply add a strongly convex quadratic function to all terms.  However, this turns out not to be necessary.  Instead, one can consider \emph{distributing} the $\tfrac12 \|w\|_2^2$ term into the remaining $p+q$ terms, which has a similar effect, at least in terms of the $w$ variables.  Merely adding and also distributing a term of the form $\tfrac12 \beta \|b - \bar{b}\|_2^2$, one finds that the following only differs from \eqref{convex_formula} by $\tfrac12 \beta \|b - \bar{b}\|_2^2$ and has the advantage that each of the $p+q$ terms is strongly convex:
\begin{multline}\label{strong_convex}
  F^{SC}(w, b, \mu, \bar w, \bar b) = C_1 \sum_{i=1}^p (\phi(1 - y_i(w^Tx_i + b), \mu) + \tfrac{1}{4pC_1} \|w\|_2^2 + \tfrac{\beta}{2pC_1} \|b - \bar{b}\|_2^2) \\
  + C_2 \sum_{i=p+1}^{p+q} (\phi(1 - \psi(w^T x_i + b, \mu), \mu) + \tfrac{\rho_i}{2} (\|w - \bar{w}\|_2^2 + \|b - \bar{b}\|_2^2) + \tfrac{1}{4qC_2} \|w\|_2^2),
\end{multline}
where $\beta \in (0,\infty)$ and $\rho_i \in (2(\|x_i\|_2^2 + 1)/\mu,\infty)$ for all $i \in \{p+1,\dots,p+q\}$.  (Again, we are assuming that the former options in \eqref{eq.smoothed_max} and \eqref{eq.smoothed_abs} are being used.  The use of alternatives would only potentially affect the allowable ranges for $\{\rho_i\}_{i=p+1}^{p+q}$.)

Formally, the theoretical benefits of having terms that are strongly convex are that they satisfy certain inequalities that, as we shall see in the next section, are required by the known theoretical convergence guarantees of IBFGS.

\begin{lemma}\label{lem.sc_bounds}
  If a function $f : {\cal R} \to \mathbb{R}$ is continuously differentiable over a sublevel set ${\cal L}_\gamma(f) := \{x \in {\cal R} : f(x) \leq \gamma\}$ for some $\gamma \in \mathbb{R}$, its gradient function is Lipschitz continuous with constant $L_{\nabla f}$ over ${\cal L}_\gamma(f)$, and $f$ is strongly convex over ${\cal L}_\gamma(f)$, then there exists $\mu_{\nabla f} \in (0,\infty)$ such that, for all $(\omega,\bar\omega) \in {\cal L}_\gamma(f) \times {\cal L}_\gamma(f)$, one has
  \begin{equation*}
    \mu_{\nabla f} \|\omega - \bar\omega \|_2^2 \leq (\nabla f(\omega) - \nabla f(\bar\omega))^T (\omega - \bar\omega) \leq L_{\nabla f} \| \omega - \bar\omega \|_2^2.
  \end{equation*}
\end{lemma}
\begin{proof}
  The former inequality follows from \cite[Proposition~4.10]{vial1983}.  The latter follows from the Lipschitz continuity of the gradient of $f$ over ${\cal L}_\gamma(f)$.
\end{proof}

\section{(Incremental) Quasi-Newton Methods}\label{IQN}

Suppose that one aims to minimize a twice-continuously differentiable objective function.  Quasi-Newton methods attempt to mimic the behavior of Newton's method for the minimization of such a function using only first-order derivative information.  This is accomplished by maintaining a matrix representing an approximation of the Hessian of the objective.  This matrix is updated iteratively using iterate and gradient displacements during the minimization algorithm, and in favorable situations the minimization algorithm can attain a superlinear rate of convergence \cite{BroDM73}.

Mathematically, a quasi-Newton strategy operates as follows.  For simplicity of notation, let us denote the objective function as $f : \mathbb{R}^{n+1} \to \mathbb{R}$ (i.e., let us update our notation so that ${\cal R} \equiv \mathbb{R}^{n+1}$ where, for our purposes later on, we replace $(w,b) \in \mathbb{R}^n \times \mathbb{R}$ in a natural manner with $\omega \equiv (w^T,b^T)^T \in \mathbb{R}^{n+1}$) and suppose that the minimization algorithm has reached iteration $k \in \mathbb{N}$ and that the current approximation is the symmetric positive definite matrix $B_k \in \mathbb{R}^{n+1} \times \mathbb{R}^{n+1}$.  It is assumed that the minimization algorithm computes a search direction $p_k$ as the minimizer of the local quadratic approximation of~$f$ at $\omega_k \in \mathbb{R}^{n+1}$ given by $\nabla f(\omega_k)^Tp + \tfrac12 p^TB_kp$, computes a step size $\alpha_k \in (0,\infty)$, then sets the next iterate as $\omega_{k+1} \gets \omega_k - \alpha_k B_k^{-1}\nabla f(\omega_k)$, as in any Newton-like scheme.  The signifying feature of a quasi-Newton strategy is the manner in which the subsequent approximation matrix $B_{k+1} \in \mathbb{R}^{n+1} \times \mathbb{R}^{n+1}$ is chosen.  In particular, in a quasi-Newton strategy, one defines the iterate and gradient displacements
\begin{equation}\label{pair}
  s_k = \omega_{k+1} - \omega_k\ \ \text{and}\ \ y_k = \nabla f(\omega_{k+1}) - \nabla f(\omega_k),
\end{equation}
respectively, then chooses $B_{k+1}$ satisfying the so-called secant equation $B_{k+1} s_k = y_k$.

The secant equation does not determine $B_{k+1}$ uniquely; there are various choices that satisfy it.  The popular and effective BFGS method makes the selection of the matrix satisfying the secant equation whose inverse minimizes the difference from the inverse of the prior approximation.  In particular, the BFGS method chooses $B_{k+1}$ as the unique solution to the optimization problem
\begin{equation*}
  \min_{B^{-1}} \|B^{-1} - B_k^{-1}\|\ \ \text{s.t.}\ \ B^{-1} = B^{-T}\ \ \text{and}\ \ B^{-1}y_k = s_k,
\end{equation*}
where the norm in the objective of this problem is a particular weighted Frobenius norm chosen such that a rank-two change in $B_k^{-1}$ yields $B_{k+1}^{-1}$ \cite{Nocedal2006}.  Applying the Sherman-Morrison-Woodbury formula to the formula for $B_{k+1}^{-1}$ in terms of $B_k^{-1}$, one obtains the following formula for $B_{k+1}$ in terms of $B_k$:
\begin{equation*}
  B_{k+1} = B_k - \frac{B_k s_k s_k^T B_k}{s_k^T B_k s_k} + \frac{y_k y_k^T}{y_k^Ts_k}.
\end{equation*}
For future reference, we remark that if the pair $(s_k,y_k)$ satisfies $s_k^Ty_k > 0$ for all $k \in \mathbb{N}$, then as long as the initial approximation $B_0$ is symmetric and positive definite (e.g., $B_0 = I$), then all elements of the sequence $\{B_k\}$ will be symmetric and positive definite as well.  Moreover, $s_k^Ty_k > 0$ is guaranteed to hold for all $k \in \mathbb{N}$ if the gradients employed in the definition of $y_k$ are computed from a strongly convex function \cite{Nocedal2006}.  

If the objective function $f$ is not smooth, but still locally Lipschitz, then while standard convergence guarantees for quasi-Newton methods no longer apply, it has been shown to be effective in practice to employ quasi-Newton methods with generalized gradients in place of gradients.  Following the developments of Clarke \cite{Clar83}, the set of generalized gradients for a locally Lipschitz function $f$ at $\omega$ is given by
\begin{equation*}
  \eth f(\omega) = \text{conv} \left( \left\{ \lim_{\bar\omega \to \omega,\ \bar\omega \in {\cal D}_f} \nabla f(\omega_k) \right\} \right),
\end{equation*}
where ${\cal D}_f \subseteq \mathbb{R}^{n+1}$ is the set of points at which $f$ is differentiable.  For many problems, at any~$\omega$, one is able to compute an (arbitrary) element from $\eth f(\omega)$, meaning that in a practical application of a quasi-Newton method one can replace $y_k$ from \eqref{pair} with
\begin{equation*}
  y_k = g_{k+1} - g_k,\ \ \text{where}\ \ (g_{k+1},g_k) \in \eth f(\omega_{k+1}) \times \eth f(\omega_k).
\end{equation*}

\subsection{I-BFGS}

Incremental algorithms, as described in Section~\ref{Literature}, attempt to exploit the finite-sum structure of an objective function by using information from only a single component function to compute a step during each iteration, cycling through the component functions using a prescribed order or strategy.  Following this idea, for an incremental quasi-Newton method for the minimization of a finite-sum objective function of the form \eqref{problem}, an idea is to store a Hessian approximation matrix for each component function individually.  Then, since new information is obtained for only a single component function in each iteration, the Hessian approximation for that component function can be updated, e.g., using BFGS, while all others are left the same \cite{Mokhtari2018}.

An I-BFGS strategy of this type can be described as follows.  For all $k \in \mathbb{N}$ and $i \in \{1,\dots,m\}$, let $z_{k,i} \in \mathbb{R}^{n+1}$ denote the latest point at which a generalized gradient has been computed for component function~$i$ up to iteration~$k$, let $v_{k,i} \in \eth f_i(z_{k,i})$ denote this latest generalized gradient for component function~$i$ up to iteration~$k$, and let $B_{k,i} \in \mathbb{R}^{n+1 \times n+1}$ denote the Hessian approximation for component function~$i$ in iteration~$k$.  With respect to these quantities, the algorithm initializes $z_{0,i} \gets \omega_0$, $v_{0,i} \in \eth f_i(z_{0,i})$, and $B_{0,i} \succ 0$ and for all $i \in \{1,\dots,m\}$.  During iteration~$k$, a new iterate is computed using information available at the start of iteration~$k$, then the algorithm chooses an index $i_k \in \{1,\dots,m\}$ corresponding to which the Hessian approximation will be updated.  In particular, the algorithm sets $z_{k+1,i_k} \gets \omega_{k+1}$, $v_{k+1,i_k} \in \eth f_{i_k}(z_{k+1,i_k})$, $s_{k,i_k} \gets z_{k+1,i_k} - z_{k,i_k}$, and $y_{k,i_k} \gets v_{k+1,i_k} - v_{k,i_k}$, then sets
\begin{equation}\label{BFGS}
  B_{k+1,i_k} \gets B_{k,i_k} - \frac{B_{k,i_k} s_{k,i_k} s_{k,i_k}^T B_{k,i_k}}{s_{k,i_k}^T B_{k,i_k} s_{k,i_k}} + \frac{y_{k,i_k} y_{k,i_k}^T}{y_{k,i_k}^T s_{k,i_k}},
\end{equation}
whereas, for all other indices, the points, generalized gradients, and Hessian approximations are unchanged, and the iterate and gradient displacements are set to zero.

In a straightforward application of this idea, in iteration $k \in \mathbb{N}$, one is left with the desire to compute the new iterate $\omega_{k+1}$ as the minimizer of $q_k : \mathbb{R}^{n+1} \to \mathbb{R}$ defined as
\begin{equation*}
  q_k(\omega) = \frac{1}{m} \sum_{i=1}^m \left( f_i(z_{k,i}) + v_{k,i}^T(w - z_{k,i})+ \frac{1}{2} (w - z_{k,i})^T B_{k,i} (w - z_{k,i}) \right),
\end{equation*}
the solution of which can be expressed as
\begin{equation*}
  \omega_{k+1} = \left( \sum_{i=1}^m B_{k,i} \right)^{-1} \left( \sum_{i=1}^m B_{k,i} z_{k,i} - \sum_{i=1}^m v_{k,i} \right).
\end{equation*}
However, forming $\sum_{i=1}^m B_{k,i}$ by summing the elements in $\{B_{k,i}\}_{i=1}^m$, then using the resulting sum to compute $\omega_{k+1}$ could be prohibitively expensive.  Instead, in a more computationally efficient approach, the matrix $(\Tilde{B}_k)^{-1} := (\sum_{i=1}^m B_{k,i})^{-1}$ can be updated directly.  In particular, following \cite{Mokhtari2018}, one can set
\begin{equation}\label{b_inverse1}
  (\Tilde{B}_{k+1})^{-1} \gets U_k + \frac{U_k (B_{k,i_k} s_{k,i_k}) (B_{k,i_k} s_{k,i_k})^T U_k}{s_{k,i_k}^TB_{k,i_k}s_{k,i_k} - (B_{k,i_k} s_{k,i_k})^T U_k (B_{k,i_k} s_{k,i_k})},
\end{equation}
where
\begin{equation}\label{b_inverse2}
  U_k \gets (\Tilde{B}_k)^{-1} + \frac{(\Tilde{B}_k)^{-1} y_{k,i_k} y_{k,i_k}^T (\Tilde{B}_k)^{-1}}{y_{k,i_k}^Ts_{k,i_k} + y_{k,i_k}^T(\Tilde{B}_k)^{-1} y_{k,i_k}}.
\end{equation}

We state a complete algorithm, which follows that in \cite{Mokhtari2018}, as Algorithm~\ref{alg:I-BFGS}.  We remark that this algorithm employs unit step sizes, although in an implementation one should consider potentially taking shorter step sizes to ensure convergence or at least good practical behavior, as we do in our numerical experiments.  We also note that Algorithm~\ref{alg:I-BFGS} skips the BFGS update in iteration $k \in \mathbb{N}$ unless the iterate and generalized gradient displacements are sufficiently large and a curvature condition holds.  This is a standard safeguard that is needed in nonconvex settings.

\begin{algorithm}[ht]
  \caption{I-BFGS Method}
  \label{alg:I-BFGS}
  \begin{algorithmic}[1]
    \Require $c \in (0,\infty)$
    \State choose $\omega_0 \in \mathbb{R}^{n+1}$.
	\State set $z_{0,i} \gets \omega_0$, $v_{0,i} \in \eth f_i(z_{0,i})$, and $B_{0,i} \gets I$ for all $i \in \{1,\dots,m\}$
	\State set $(\Tilde{B}_0)^{-1} \gets (\sum_{i=1}^m B_{0,i})^{-1}$, $u_0 \gets \sum_{i=1}^m B_{0,i} \omega_0$, and $g_0 = \sum_{i=1}^m v_{0,i}$
	\For{$k \in \mathbb{N}$} 
	  \State set $\omega_{k+1} \gets (\Tilde{B}_k)^{-1}(u_k-g_k)$
	  \State choose $i_k \in \{1,\dots,m\}$
	  \State set $z_{k+1,i_k} \gets \omega_{k+1}$ and $z_{k+1,i} \gets z_{k,i}$ for all $i \in \{1,\dots,m\} \setminus \{i_k\}$
	  \State set $v_{k+1,i_k} \in \eth f_{i_k}(z_{k+1,i_k})$ and $v_{k+1,i} \gets v_{k,i}$ for all $i \in \{1,\dots,m\} \setminus \{i_k\}$ \label{step.v}
	  \State set $s_{k,i} \gets z_{k+1,i} - z_{k,i}$ for all $i \in \{1,\dots,m\}$ \label{step.s}
	  \State set $y_{k,i} \gets v_{k+1,i} - v_{k,i}$ for all $i \in \{1,\dots,m\}$ \label{step.y}
	  \If {$s_{k,i_k}^T y_{k,i_k} > c \|s_{k,i_k}\|_2 \|y_{k,i_k}\|_2$, $\|s_{k,i_k}\|_2 > c$, and $\|y_{k,i_k}\|_2 > c$}
	    \State set $B_{k+1,i_k}$ by \eqref{BFGS} and $B_{k+1,i} \gets B_{k,i}$ for all $i \in \{1,\dots,m\} \setminus \{i_k\}$
		\State set $(\Tilde{B}_{k+1})^{-1}$ by \eqref{b_inverse1}--\eqref{b_inverse2}.
	  \Else
	    \State set $B_{k+1,i} \gets B_{k,i}$ for all $i \in \{1,\dots,m\}$
	    \State set $(\Tilde{B}_{k+1})^{-1} \gets (\Tilde{B}_k)^{-1}$
	  \EndIf
	  \State set $u_{k+1} \gets u_k + (B_{k+1,i_k} \omega_{k+1} - B_{k,i_k} z_{k,i_k})$
	  \State set $g_{k+1} \gets g_k + v_{k+1,i_k} - v_{k,i_k}$
	\EndFor
  \end{algorithmic}
\end{algorithm}

In the following subsections, we state our proposed modifications of I-BFGS that exploit the problem reformulations and approximations described in Section~\ref{SSL}.

\subsection{I-BFGS-DC for DC Function Reformulation}

As mentioned in Section~\ref{sec.DC}, the structure of a DC function can be exploited in the context of an optimization algorithm.  For example, in the context of an I-BFGS strategy, one can exploit the properties of a DC function in an attempt to avoid the skipping of updates.  Specifically, suppose that $f_i$ is a DC function that can be expressed as $f_i = g_i - h_i$, where $g_i$ and $h_i$ are convex functions.  If, corresponding to a pair of points $(\omega,\bar\omega)$, one finds that the displacements $s = \bar\omega - \omega$ and $y = v_{\bar\omega} - v_\omega$, where $(v_\omega,v_{\bar\omega}) \in \eth f_i(\omega) \times \eth f_i(\bar\omega)$, yield $s^Ty < 0$, then a rule such as that employed in Algorithm~\ref{alg:I-BFGS} would skip the BFGS update.  This might be avoided by replacing $y$ by an approximate generalized gradient displacement, where the approximation exploits knowledge of the DC function structure of $f_i$.

The strategy that we propose is the following.  Suppose with respect to $\omega$ one has
\begin{equation*}
  v_{g_i,\omega} \in \eth g_i(\omega),\ \ v_{h_i,\omega} \in \eth h_i(\omega),\ \ \text{and}\ \ v_\omega = v_{g_i,\omega} - v_{h_i,\omega}.
\end{equation*}
Then, instead of evaluating $v_{\bar\omega}$ in a similar manner using a generalized gradient of $h_i$ at $\bar\omega$ (i.e., a subgradient, since $h_i$ is convex), one can instead employ a generalized gradient of the affine underestimator of $h_i$ defined by $h_i(\omega) + v_{h_i,\omega}^T(\cdot)$, leading to
\begin{equation*}
  v_{g_i,\bar\omega} \in \eth g_i(\bar\omega),\ \ v_{h_i,\bar\omega} \gets v_{h_i,\omega},\ \ \text{and}\ \ v_{\bar\omega} = v_{g_i,\bar\omega} - v_{h_i,\bar\omega},
\end{equation*}
which in turn leads to the approximate displacement
\begin{equation*}
  y = v_{\bar\omega} - v_\omega = v_{g_i,\bar\omega} - v_{g_i,\omega}.
\end{equation*}
Since $g_i$ is convex, this means that $s^Ty \geq 0$, and it is more likely that $s^Ty > 0$.

The algorithm that we refer to as I-BFGS-DC is identical to I-BFGS (Algorithm~\ref{alg:I-BFGS}), except that each instance of $\eth f_i(\cdot)$ is replaced by $\eth g_i(\cdot) - \eth h_i(\cdot)$ (where ``$-$'' denotes a Minkowski difference) in all places besides the computation of $v_{k+1,i_k}$ in Step~\ref{step.v}, where it is replaced by the subroutine stated in Algorithm~\ref{alg:I-BFGS-DC}.  We include the \textbf{if} condition in the subroutine since this alternative computation of $v_{k+1,i_k}$ is not needed when~$f_i$ is known to be convex.  For example, in I-BFGS-DC employed to minimize \eqref{semisupervised}, this modification would only be employed if $i_k$ corresponds to one of the final $q$ terms; it is not needed for any of the other terms since they are convex.  In any case, after the new BFGS approximation is set, the subroutine states that the value of $v_{k+1,i_k}$ is ``reset'' (if necessary) to an ``unmodified'' value for use in later iterations.

\begin{algorithm}[ht]
  \caption{I-BFGS-DC modified of computation of $v_{k+1,i_k}$}
  \label{alg:I-BFGS-DC}
  \begin{algorithmic}
    \If {$f_{i_k}$ is nonconvex}
      \State set $v_{k+1,i_k} \in \eth g_{i_k}(z_{k+1,i_k}) - \eth h_{i_k}(z_{k,i_k})$
    \Else
      \State set $v_{k+1,i_k} \in \eth g_{i_k}(z_{k+1,i_k}) - \eth h_{i_k}(z_{k+1,i_k})$
    \EndIf
    \State {}[... all proceeds as in I-BFGS until after line 17, then...]
    \If {$f_{i_k}$ is nonconvex}
      \State set $v_{k+1,i_k} \in \eth g_{i_k}(z_{k+1,i_k}) - \eth h_{i_k}(z_{k+1,i_k})$
    \EndIf 
  \end{algorithmic}
\end{algorithm}

\subsection{I-BFGS-S for Smooth Approximation}

I-BFGS can be applied to a smooth approximation of \eqref{problem} in a straightforward manner.  In particular, for all $i \in \{1,\dots,m\}$, letting $f_i^S(\cdot,\mu_i) : \mathbb{R}^{n+1} \to \mathbb{R}$ denote a smoothing function approximation of $f_i$ given a smoothing parameter $\mu_i \in (0,\infty)$, problem~\eqref{problem} (with ${\cal R} = \mathbb{R}^{n+1}$, as has been used in this section) is approximated by
\begin{equation}\label{problem_smooth}
  \min_{\omega \in \mathbb{R}^{n+1}} \tfrac{1}{m} \sum_{i=1}^m f_i^S(\omega,\mu_i).
\end{equation}
The algorithm that we refer to as I-BFGS-S is identical to I-BFGS (Algorithm~\ref{alg:I-BFGS}) except that the algorithm initializes $\mu_{0,i} \gets \mu$ for all $i \in \{1,\dots,m\}$ for some $\mu \in (0,\infty)$, each instance of an element of $\eth f_i(\cdot)$ is replaced by $\nabla f_i(\cdot,\mu_{k,i})$, and the smoothing parameter update in Algorithm \ref{alg:I-BFGS-smoothing} is added immediately after Step~\ref{step.s}; see~\cite{chen}.

\begin{algorithm}[ht]
  \caption{I-BFGS-S addition for updating smoothing parameters}
  \label{alg:I-BFGS-smoothing}
  \begin{algorithmic}
    \Require $\kappa \in (0,\infty)$, $\sigma \in (0,1)$
    \State set $\mu_{k+1,i} \gets \mu_{k,i}$ for all $i \in \{1,\dots,m\} \setminus \{i_k\}$
    \If {$\|v_{k+1,i_k}\|_2 \geq \kappa \mu_{k,i_k}$} 
      \State set $\mu_{k+1,i_k} \gets \mu_{k,i_k}$
    \Else 
      \State set $\mu_{k+1,i_k} \gets \sigma \mu_{k,i_k}$ and replace $v_{k+1,i_k} \gets \nabla f_{i_k}^S(z_{k+1,i_k}, \mu_{k+1,i_k})$
    \EndIf
  \end{algorithmic}
\end{algorithm}

We remark that since the gradients computed in I-BFGS-S depend on the smoothing parameters, it is possible that a gradient displacement---employed in a BFGS update---involves a difference of gradients computed using two different values of the smoothing parameter.  We did not find this to be an issue in our experiments, but it is something of which one should be aware in the use of I-BFGS-S.

\subsection{I-BFGS-C for Smooth Convex Local Approximation}

Suppose that, for all $i \in \{1,\dots,m\}$, the function $f_i^S(\cdot,\mu_i)$ is a continuously differentiable and $\rho_i$-convex approximation of $f_i$.  For example, for the function $F^S$ in \eqref{smooth}, the first term is continuously differentiable and $1/2$-convex, the next $p$ terms are continuously differentiable and $0$-convex, and the remaining $q$ terms are continuously differentiable and $\rho_i$-convex for sufficiently negative $\rho_i$, as shown in Section~\ref{sec.convex}.  Then, in the context of an I-BFGS-type algorithm, rather than compute the gradient displacement for the function with index $i_k$ as in I-BFGS-S, one can attempt to avoid the need to skip an update by computing the displacement according to a smooth convex local approximation of $f_{i_k}$.  This can be done following the discussion in Section~\ref{sec.convex} with the ``base point'' always set as the last iterate at which a gradient was computed.

The algorithm that we refer to as I-BFGS-C is identical to I-BFGS (Algorithm~\ref{alg:I-BFGS}) except that the algorithm initializes $\mu_{0,i} \gets \mu$ for all $i \in \{1,\dots,m\}$ for some $\mu \in (0,\infty)$, each instance of an element of $\eth f_i(\cdot)$ is replaced by $\nabla f_i(\cdot,\mu_{k,i})$, and Step~\ref{step.y} is replaced by the subroutine stated as Algorithm~\ref{alg:I-BFGS-smoothing-convex}.  This subroutine updates the smoothing parameters using the same strategy as in I-BFGS-S, and modifies the update for the gradient displacement if the smooth and $\rho_{i_k}$-convex approximation of $f_{i_k}$ is nonconvex.

\begin{algorithm}[ht]
  \caption{I-BFGS-C additions}
  \label{alg:I-BFGS-smoothing-convex}
  \begin{algorithmic}[1]
    \Require $\kappa \in (0,\infty)$, $\sigma \in (0,1)$
    \State set $\mu_{k+1,i} \gets \mu_{k,i}$ for all $i \in \{1,\dots,m\} \setminus \{i_k\}$
	\State set $y_{k,i} \gets v_{k+1,i} - v_{k,i}$ for all $i \in \{1,\dots,m\} \setminus \{i_k\}$
    \If {$\|v_{k+1,i_k}\|_2 \geq \kappa \mu_{k,i_k}$}
	  \State set $\mu_{k+1,i_k} \gets \mu_{k,i_k}$
	\Else
	  \State set $\mu_{k+1,i_k} \gets \sigma \mu_{k,i_k}$ and replace $v_{k+1,i_k} \gets \nabla f_{i_k}^S (z_{k+1,i_k},\mu_{k+1,i_k})$
	\EndIf
	\If{$f_{i_k}^S(\cdot,\mu_{k+1,i_k})$ is convex}
	  \State set $y_{k,i_k} \gets v_{k+1,i_k} - v_{k,i_k}$
	\Else
	  \State set $\rho$ such that $f_{i_k}^S(\cdot,\mu_{k+1,i_k}) + \rho \|\cdot\|_2^2$ is convex \label{line.rho}
	  \State set $y_{k,i_k} \gets v_{k+1,i_k} - v_{k,i_k} + 2 \rho (z_{k+1,i_k} - z_{k,i_k})$
	\EndIf
  \end{algorithmic}
\end{algorithm}

\subsection{I-BFGS-SC for Smooth Strongly Convex Local Approximation}

The algorithm that we refer to as I-BFGS-SC is identical to I-BFGS-C except that, in line~\ref{line.rho} of Algorithm~\ref{alg:I-BFGS-smoothing-convex}, the value for $\rho$ is chosen such that the stated function is strongly convex.  In addition, in the particular application of I-BFGS-SC for the minimization of $F$ in \eqref{semisupervised} that we explore through numerical experimentation in the following section, we employ the distribution of terms as described in Section~\ref{sec.strongly_convex}.

In general, I-BFGS-type methods have no convergence guarantees when employed to minimize a function that is nonsmooth and nonconvex.  However, using the theoretical guarantees from \cite{Mokhtari2018}, one can establish a convergence guarantee for the minimization of an approximate objective if, after some iteration, the smoothing functions that are employed yield twice-continuously differentiable functions with Lipschitz Hessians, the smoothing parameters are no longer modified, the ``base point'' used for convexification is no longer changed, and in place of gradients of $f_i^S(\cdot,\mu_i)$ one uses gradients of the convexified terms.  In other words, one can establish a convergence guarantee in a setting in which, after some iteration, the I-BFGS strategy is employed to minimize a fixed function with strongly convex and sufficiently smooth terms.  The function being minimized in this setting is not the original nonconvex and nonsmooth objective function, but if the ``base point'' is close to a minimizer of the original objective function, then the minimizer of the employed strongly convex approximation may be close (or closer) to a minimizer of the original objective.

Under the conditions stated in the previous paragraph (and formalized in the statement of the theorem below), one can establish the following convergence guarantee for I-BFGS \cite{Mokhtari2018}, which we state here for convenience.  For simplicity of notation, we write the guarantee in terms of the minimization of $\tilde{f} = \tfrac1m \sum_{i=1}^m \tilde{f}_i$, where $\tilde{f}_i \approx f_i$ for all $i \in \{1,\dots,m\}$ such that $\{\tilde{f}_i\}_{i=1}^m$ satisfies the following assumption.  Recall that we have already established in Lemma~\ref{lem.sc_bounds} a set of conditions---including strong convexity---that ensure the first set of bounds in the assumption.  The final inequality in the assumption can be satisfied if one employs certain smoothing functions.

\begin{assumption}\label{ass.theory}
  There exist $\mu_{\nabla f} \in (0,\infty)$, $L_{\nabla f} \in (0,\infty)$, and $L_{\nabla^2 f} \in (0,\infty)$ such that, for all $i \in \{1,\dots,m\}$ and $(\omega,\bar\omega)$, one has
  \begin{equation*}
    \begin{aligned}
      \mu_{\nabla f} \|\omega - \bar\omega \|_2^2 \leq (\nabla \tilde{f}_i(\omega) - \nabla \tilde{f}_i(\bar\omega))^T (\omega - \bar\omega) &\leq L_{\nabla f} \| \omega - \bar\omega \|_2^2 \\
      \text{and}\ \ \|\nabla^2 \tilde{f}_i(\omega) - \nabla^2 \tilde{f}_i(\bar\omega)\|_2 &\leq L_{\nabla^2 f} \|\omega - \bar\omega\|_2.
    \end{aligned}
  \end{equation*}
\end{assumption}

Since $\tilde{f}$ is strongly convex, it has a unique global minimizer $\omega_*$.

\begin{lemma}\label{lemma1} (\cite[Lemma 1]{Mokhtari2018})
  Suppose that $i_k = (k\ {\rm mod}\ n) + 1$ for all $k \in \mathbb{N}$ and define $\sigma_k := \max\{\|z_{k+1,i_k} - \omega_*\|_2, \|z_{k,i_k} - \omega_*\|_2\}$ and $M_{i_k} := \nabla^2 f_{i_k} (\omega_*)^{-1/2}$ for all $k \in \mathbb{N}$.  If Assumption~\ref{ass.theory} holds and $\sigma_k < m/(3L_{\nabla^2 f})$, then
  \begin{equation*}
    \|B_{k+1,i_k} - \nabla^2 f_{i_k}(\omega_*)\|_{M_{i_k}} \leq \left((1 - \alpha \theta_k^2 )^{1/2} + \alpha_1 \sigma_k \right) \|B_{k,i} - \nabla^2 f_{i_k}(\omega_*)\|_{M_{i_k}} + \alpha_2 \sigma_k
  \end{equation*}
  for some $(\alpha, \alpha_1, \alpha_2) \in (0,\infty)^3$, where
  \begin{equation*}
    \theta_k := \begin{cases} \frac{\|M_{i_k}(B_{k,i_k} - \nabla^2 f_{i_k}(\omega_*)) s_{k,i_k}\|_2}{\|B_{k,i_k} - \nabla^2 f_{i_k}(\omega_*) \|_{M_{i_k}} \|M_{i_k}^{-1} s_{k,i}\|_2} & \text{if $B_{k,i_k} \neq \nabla^2 f_{i_k}(\omega_*)$} \\ 0 & \text{otherwise.} \end{cases}
  \end{equation*}
\end{lemma}

\begin{lemma}\label{lemma3} (\cite[Lemma 3]{Mokhtari2018})
  Suppose that $i_k = (k\ {\rm mod}\ n) + 1$ for all $k \in \mathbb{N}$ and define $M_i := \nabla^2 f_i(\omega_*)^{-1/2}$ for all $i \in \{1,\dots,m\}$.  If Assumption~\ref{ass.theory} holds, then, for any $r \in (0,1)$, there exist $\epsilon(r) \in (0,\infty)$ and $\delta(r) \in (0,\infty)$ such that having both $\|\omega_0 - \omega_*\|_2 < \epsilon(r)$ and $\|B_{0,i} - \nabla^2 f_i(\omega_*)\|_{M_i} < \delta(r)$ for all $i \in \{1,\dots,m\}$ implies that
  \begin{equation*}
    \|w_k - \omega_*\|_2 \leq r^{[\frac{k-1}{m}+1]+1} \|\omega_0-\omega_*\|_2.
  \end{equation*}
  In addition, the sequences $\{\|B_{k,i}\|_2\}$ and $\{\|B_{k,i}^{-1}\|_2\}$ are uniformly bounded. 
\end{lemma}

\begin{theorem} (\cite[Theorem 7]{Mokhtari2018})
  Suppose that the conditions in Lemmas~\ref{lemma1} and \ref{lemma3} hold and define $\tilde\omega_k := argmax_{j \in \{k n,\dots,kn + n - 1\}} \{\|\omega_j - \omega_*\|_2\}$.  Then, $\{\tilde\omega_k\}$ converges to~$\omega_*$ at a superlinear rate in the sense that
  \begin{equation*}
    \lim_{k \to \infty} \frac{\|\tilde\omega_{k+1} - \omega_*\|_2}{\|\tilde\omega_k - \omega_*\|_2} = 0.
  \end{equation*}
  In addition, there exists $\{\zeta_k\}$ such that $\|w_k - \omega_*\|_2 \leq \zeta_k$ for all $k \in \mathbb{N}$, where $\{\zeta_k\}$ vanishes at a superlinear rate in the sense that
  \begin{equation*}
    \lim_{k \to \infty} \frac{\zeta_{k+1}}{\zeta_k} = 0.
  \end{equation*}
\end{theorem}

\section{Computational Results}\label{computational}

We implemented I-BFGS, I-BFGS-DC, I-BFGS-S, I-BFGS-C, and I-BFGS-SC as described in the previous section to minimize the semi-supervised machine learning objective function~\eqref{semisupervised} using the data sets that are summarized in Table~\ref{table: problems} below.  The data sets that are numbered 1--6 and 9--12 are from \cite{murphy1994} and \cite{chang2011}, respectively, while the data sets that are numbered 7--8 were generated as described in \cite{chapelle2005}.  We ran the experiments using the polyps cluster in the COR@L Laboratory at Lehigh University.\footnote{https://coral.ise.lehigh.edu/wiki/doku.php/info:coral}  For the data sets numbered 1--8, ten-fold cross-validation was performed to split the data into training and testing sets.  For the data sets numbered 9--11, the data was divided into training and testing sets as described in \cite{chang2011}.  For the data set numbered~12, the data was split into training and testing sets randomly, as in \cite{astorino2015}. 

\begin{table}[ht]
\caption{Data Sets}
	{\begin{tabular}{p{0.1cm}p{1.8cm}p{1cm}p{1.2cm}p{1.2cm}||p{0.1cm}p{1.8cm}p{1cm}p{1.2cm}p{1.2cm}} 
		\hline \multirow{3}{0.8 in}{$\#$} & \multirow{3}{0.8 in} {Data Set} & \multirow{3}{0.8 in}{$n$} & \multicolumn{2}{c} {$m$} & \multirow{3}{0.8 in}{ $\#$} & \multirow{3}{0.8 in}{ Data Set} & \multirow{3}{0.8 in}{$n$} & \multicolumn{2}{c} {$m$} \\
		\cline{4-5} \cline{9-10}& & & Training & Test & & & & Training & Test\\
		\hline
1 & Ionosphere & 34 & 315 & 35 & 9 & a9a     & 123 & 32562  & 16561  \\
2 & Pima       & 9  & 690 & 77 & 10 & w8a     & 300 & 49749  & 14951  \\
3 & Sonar      & 60 & 186 & 21 & 11 & ijcnn1  & 22  & 49990  & 91701  \\
4 & Diagnostic & 30 & 511 & 57 & 12 & Covtype & 54  & 464810 & 116202 \\
5 & Heart      & 13 & 243 & 27 &    &         &     &        &        \\
6 & Cancer     & 9  & 614 & 69 &    &         &     &        &        \\
7 & g50c       & 50 & 495 & 55 &    &         &     &        &        \\
8 & g10n       & 10 & 495 & 55 &    &         &     &        &       \\
\hline
	\end{tabular}}
\label{table: problems}
\end{table}

\subsection{Parameters of Algorithms}

For each data set, the selection of the values $C_1$ and $C_2$ was performed as in \cite{astorino2015}.  In particular, using the training set for each data set, a training phase was performed that considered the values $C_1=10^i$ for each $i \in \{-1,0,1,2\}$ and $C_2=C_1^{-j}$ for each $j \in \{0,1,2\}$.  Using these results, the values of $C_1$ and $C_2$ were set (separately for each data set) as those yielding the best results for the testing data set for each data set.

For all variants of I-BFGS, the iteration limit was set as $10^4$ for data sets 1--8 and $2m$ for data sets 9--12.  In all cases, the starting point was generated randomly in the interval $[-5,5]$ for each component of $(w,b)=(w_1,...,w_n,b)$ and the value $c$ was set at $10^{-8}.$   For I-BFGS-S, I-BFGS-C, and I-BFGS-SC, the parameters $\mu_{0,i}$ for all $i \in \{1,\dots,m\}$, $\sigma$, and $\kappa$ were set as $0.1$, $0.9$, and $0.5$, respectively.  Lastly, for I-BFGS-C and I-BFGS-SC, the $\rho$ values were set as described Sections~\ref{sec.convex} and \ref{sec.strongly_convex}.

For the purposes of comparison, we also implemented the bundle method, referred to as TSVM-Bundle, with the method-specific parameters as described in \cite{astorino2015,fuduli2004}.  The iteration limit and starting points were set in the same manner as for the I-BFGS methods.  The subproblems for TSVM-Bundle were solved using Gurobi \cite{gurobi}.

\subsection{Comparison of Algorithms}

Table \ref{table:results} provides the results for all algorithms applied to each problem instance in terms of average testing error.  For each data set, we consider different percentages of labeled data (indicated in the ``\% Labeled Data'' column) to show results across a wide spectrum of the balance of terms in the objective \eqref{semisupervised}.  Since ours are only prototype implementations of the methods, we do not provide detailed time comparisons of the runs.  That said, we mention in passing that the computational time for the I-BFGS methods was always significant less than the computational times for TSVM-Bundle.  For a glimpse of this fact, we note that the average times for the I-BFGS methods for data sets 1--8 was around 1.8 seconds and for data sets 9--12 was around 50 seconds, whereas for TSVM-Bundle the average times was on the order of $10^3$ seconds for data sets 1--8 and on the order of $10^5$ seconds for data sets 9--12.

To provide an easier-to-visualize summary of the performance of the algorithms, we also provide performance profiles \cite{dolan2002}.  A performance profile makes use of a performance measure $t_{p,s}$, performance ratio $r_{p,s}$, and (cumulative) fractional performance $\rho_{s}(\tau)$, where $p$ and $s$ are indices for problem and algorithm, respectively.  For our comparison, we the relevant improvement of results for the performance measurement (see e.g. \cite{ali2005,beiranvand2017,vaz2007}), specifically,
\begin{equation*}
  t_{p,s}=1+\frac{f_{p,s}-f^b_{p}}{f^w_p-f^b_p},
\end{equation*}
where $f_{p,s}$ is the objective value of problem $p$ obtained by algorithm $s$, and $f^b_p$ and $f^w_p$ are the best and the worst final objective values of problem $p$ obtained among all algorithms, respectively.  Generally speaking, a performance profile shows how well each algorithm performs relative the others.  In each profile, a plotted point $(\tau,\rho_{s}(\tau))$ says that algorithm $s$ finds solutions within $\tau$ times of the best found solution in $\rho_{s}(\tau)*100\%$ of the problems. 

The profile in Figure~\ref{figure:performance_profile} shows that the I-BFGS algorithms all outperform TSVM-Bundle in general, with I-BFGS-SC yielding the best results overall.  Since, as previously mentioned, the I-BFGS methods require significantly less CPU time, our results show overall that I-BFGS methods are preferable in this setting.


\begin{longtable}{p{0.1cm}p{1cm}p{1.4cm}p{1.4cm}p{1.4cm}p{1.4cm}p{1.4cm}p{1.4cm}}
\caption{Average testing error of I-BFGS, I-BFGS-DC, I-BFGS-S, I-BFGS-C, I-BFGS-SC and TSVM-Bundle}\\
\hline
\tiny \# & \tiny \% Labeled Data & \tiny I-BFGS & \tiny I-BFGS-DC & \tiny I-BFGS-S & \tiny I-BFGS-C & \tiny I-BFGS-SC & \tiny TSVM-Bundle \\ \hline
\endfirsthead
\tiny \# & \tiny \% Labeled Data & \tiny I-BFGS & \tiny I-BFGS-DC & \tiny I-BFGS-S & \tiny I-BFGS-C & \tiny I-BFGS-SC & \tiny TSVM-Bundle \\ \hline
\endhead
\multirow{10}{*}{1}  & 10\%  & \textbf{17.71} & 21.14          & 21.14          & 29.43          & 21.14          & 19.14          \\
                     & 20\%  & 16.85          & \textbf{14.28} & 22             & 18.29          & 16.86          & 19.71          \\
                     & 30\%  & 14.57          & \textbf{4.57}  & 13.71          & 14.86          & 15.14          & 13.42          \\
                     & 40\%  & 17.99          & \textbf{14}    & 16.28          & 15.71          & 18.29          & 16             \\
                     & 50\%  & 13.71          & 13.99          & 12.85          & \textbf{12.57} & \textbf{12.57} & 14.28          \\
                     & 60\%  & 12.57          & 12.57          & 12             & 12             & 11.71          & \textbf{11.42} \\
                     & 70\%  & \textbf{11.41} & 13.42          & 13.14          & 13.43          & 12.29          & 13.71          \\
                     & 80\%  & 11.71          & 12             & 12.57          & 12             & \textbf{11.43} & 12.57          \\
                     & 90\%  & 12.28          & 11.71          & \textbf{10.57} & 11.14          & 11.14          & 11.14          \\
                     & 100\% & \textbf{10.85} & 12.85          & 11.71          & 12.29          & 12.57          & \textbf{10.85} \\
\hline
\multirow{10}{*}{2}  & 10\%  & 26.59          & 26.84          & 27.11          & 26.2           & \textbf{23.59} & 24.89          \\
                     & 20\%  & 24.75          & 24.23          & 23.71          & 23.97          & 24.23          & \textbf{23.7}  \\
                     & 30\%  & 22.93          & \textbf{22.15} & 23.06          & 23.32          & 22.54          & 22.66          \\
                     & 40\%  & 21.75          & 22.14          & 22.02          & 22.54          & 21.64          & \textbf{21.63} \\
                     & 50\%  & 22.02          & \textbf{21.49} & 22.02          & 22.81          & 22.01          & 21.75          \\
                     & 60\%  & 20.72          & \textbf{20.45} & 21.5           & 21.63          & 21.11          & 21.63          \\
                     & 70\%  & 21.89          & 21.24          & 21.5           & \textbf{20.85} & 22.03          & 21.89          \\
                     & 80\%  & 22.15          & \textbf{20.85} & 22.68          & 22.29          & 22.02          & 21.63          \\
                     & 90\%  & 22.02          & \textbf{20.45} & 22.15          & 22.29          & 21.51          & 21.89          \\
                     & 100\% & 21.89          & 22.67          & \textbf{21.76} & 22.02          & 22.29          & 22.28          \\
\hline
\multirow{10}{*}{3}  & 10\%  & 36.26          & 40.54          & 32.92          & 22.69          & \textbf{16.74} & 17.21          \\
                     & 20\%  & 23             & \textbf{13.92} & 27.66          & 20.45          & 20.93          & 23.14          \\
                     & 30\%  & 13.16          & \textbf{8.76}  & 9.11           & 27.4           & 15.93          & 10.16          \\
                     & 40\%  & 19.52          & \textbf{8.14}  & 17.42          & 13.4           & 15.6           & 29.71          \\
                     & 50\%  & \textbf{7.21}  & 20.47          & 24.35          & 27.9           & 18.93          & 29.95          \\
                     & 60\%  & \textbf{10.07} & 18.47          & 18.78          & 11.5           & 20.31          & 20.59          \\
                     & 70\%  & 26.14          & 21.42          & \textbf{13.02} & 13.5           & 15.93          & 17.95          \\
                     & 80\%  & 16.19          & 18             & 20.47          & \textbf{15.98} & 26.52          & 24.63          \\
                     & 90\%  & 25.4           & \textbf{13.5}  & 32.02          & 17.31          & 17.1           & 23.5           \\
                     & 100\% & \textbf{21.26} & 32             & 32             & 31.6           & 29.1           & 37.38          \\
\hline
\multirow{10}{*}{4}  & 10\%  & 14.1           & \textbf{7.94}  & 14.63          & 11.46          & 12.7           & 11.83          \\
                     & 20\%  & 7.58           & 5.46           & 9.01           & 6.19           & 5.47           & \textbf{4.06}  \\
                     & 30\%  & 4.48           & 4.4            & 4.41           & 4.59           & 4.76           & \textbf{4.06}  \\
                     & 40\%  & 4.04           & 3.69           & 4.58           & 3.7            & 2.99           & \textbf{2.28}  \\
                     & 50\%  & 2.98           & 2.99           & 2.28           & 3.17           & 2.99           & \textbf{1.93}  \\
                     & 60\%  & 3.58           & 2.98           & 3.33           & 2.99           & \textbf{2.28}  & \textbf{2.28}  \\
                     & 70\%  & 2.81           & 3.33           & \textbf{2.1}   & 2.99           & 3.17           & 2.11           \\
                     & 80\%  & 3.16           & 3.34           & 3.16           & 2.29           & 2.29           & \textbf{2.1}   \\
                     & 90\%  & 2.1            & 2.45           & 2.63           & \textbf{1.76}  & 2.46           & 2.45           \\
                     & 100\% & 2.28           & 2.81           & 2.28           & 3.34           & 2.28           & \textbf{2.11}  \\
\hline
\multirow{10}{*}{5}  & 10\%  & 28.51          & \textbf{20.74} & 23.7           & 24.44          & 25.19          & 25.92          \\
                     & 20\%  & 19.62          & 19.25          & 20.37          & 22.96          & 22.22          & \textbf{18.88} \\
                     & 30\%  & 15.92          & 15.18          & 14.81          & 17.78          & \textbf{14.44} & 15.18          \\
                     & 40\%  & 16.66          & \textbf{13.33} & 16.66          & 15.19          & 15.93          & 16.29          \\
                     & 50\%  & 14.44          & \textbf{13.7}  & 15.55          & \textbf{13.7}  & 14.81          & 15.92          \\
                     & 60\%  & 15.55          & \textbf{13.7}  & 14.44          & 14.81          & 15.19          & 14.44          \\
                     & 70\%  & 14.44          & 14.07          & \textbf{13.7}  & 14.44          & 14.07          & 14.44          \\
                     & 80\%  & \textbf{12.22} & \textbf{12.22} & 14.07          & 12.96          & \textbf{12.22} & 13.33          \\
                     & 90\%  & 14.07          & 13.33          & 15.18          & \textbf{12.59} & 12.96          & 12.96          \\
                     & 100\% & \textbf{14.44} & \textbf{14.44} & 14.81          & 14.81          & 15.56          & \textbf{14.44} \\
\hline
\multirow{10}{*}{6}  & 10\%  & 3.22           & 3.66           & \textbf{2.63}  & 3.22           & 2.64           & 2.92           \\
                     & 20\%  & 4.23           & 2.92           & \textbf{2.63}  & 2.64           & 2.78           & 2.92           \\
                     & 30\%  & 3.36           & 3.21           & 3.21           & \textbf{2.63}  & 2.78           & 3.5            \\
                     & 40\%  & 3.65           & 3.21           & 3.2            & \textbf{2.2}   & 2.77           & 3.35           \\
                     & 50\%  & 3.21           & 2.77           & 3.06           & \textbf{2.2}   & 2.92           & 2.92           \\
                     & 60\%  & 3.35           & 2.77           & 2.63           & \textbf{2.34}  & 2.92           & 2.77           \\
                     & 70\%  & 3.36           & 3.49           & 2.92           & 2.63           & \textbf{2.34}  & 2.77           \\
                     & 80\%  & 2.63           & 3.06           & 2.77           & \textbf{2.48}  & 3.36           & 2.63           \\
                     & 90\%  & 3.21           & 2.92           & 2.48           & \textbf{2.05}  & 2.63           & 3.07           \\
                     & 100\% & \textbf{2.48}  & \textbf{2.48}  & 3.21           & 2.92           & 3.06           & 2.63           \\
\hline
\multirow{10}{*}{7}  & 10\%  & 30             & 23.63          & \textbf{0.54}  & 10.91          & 36.91          & 13.45          \\
                     & 20\%  & 16.18          & 29.27          & 28.18          & 19.09          & \textbf{14.91} & 28.54          \\
                     & 30\%  & 39.27          & 18             & \textbf{9.63}  & 38             & 33.45          & 39.45          \\
                     & 40\%  & 33.81          & 27.27          & 19.81          & 30             & \textbf{10}    & 23.45          \\
                     & 50\%  & 29.81          & 10             & 20.9           & 29.09          & \textbf{3.45}  & 24             \\
                     & 60\%  & 13.09          & \textbf{6.9}   & 26             & 12.36          & 24.18          & 30.54          \\
                     & 70\%  & 7.09           & \textbf{6.18}  & 8.36           & 7.45           & 6.73           & 8.72           \\
                     & 80\%  & \textbf{3.63}  & 6.72           & \textbf{3.63}  & 3.64           & 3.82           & 7.27           \\
                     & 90\%  & 4.18           & 4.36           & 4.9            & 5.82           & \textbf{4}     & 4.72           \\
                     & 100\% & 4.9            & \textbf{1.63}  & 6              & 6.18           & 6.73           & 6.72           \\
\hline
\multirow{10}{*}{8}  & 10\%  & 12.18          & 12.54          & 12             & 12.73          & \textbf{10.36} & 12.36          \\
                     & 20\%  & 6.9            & 4.72           & 5.27           & 4.91           & \textbf{2.73}  & 5.09           \\
                     & 30\%  & 4.54           & 3.45           & 4.36           & 3.45           & \textbf{2.36}  & 6.18           \\
                     & 40\%  & 3.45           & 3.09           & 3.27           & 3.27           & \textbf{2.36}  & 3.45           \\
                     & 50\%  & 2.36           & 1.45           & 2.54           & 2.18           & \textbf{1.27}  & 3.09           \\
                     & 60\%  & 1.45           & 1.81           & 1.27           & 1.45           & \textbf{0.91}  & 1.99           \\
                     & 70\%  & 1.63           & 1.63           & 1.63           & 1.82           & \textbf{1.45}  & 2              \\
                     & 80\%  & 1.81           & \textbf{1.63}  & 2.18           & 2.18           & 1.64           & 2.36           \\
                     & 90\%  & 2              & \textbf{0.54}  & 2              & 1.45           & 1.09           & 3.63           \\
                     & 100\% & 1.47           & 1.81           & 1.63           & 1.45           & \textbf{0.73}  & 2.36           \\
\hline
\multirow{10}{*}{9}  & 10\%  & 23.78          & 21.9           & 21.09          & 23.63          & 19.3           & \textbf{19.02} \\
                     & 20\%  & 20.89          & 23.63          & \textbf{17.23} & 23.54          & 23.6           & 15.86          \\
                     & 30\%  & 20.42          & 34.46          & 18.19          & \textbf{17.03} & 23.43          & 24             \\
                     & 40\%  & \textbf{17.36} & 20.59          & 20.57          & 23.63          & 17.57          & 17.72          \\
                     & 50\%  & \textbf{15.24} & 23.11          & 23.63          & 23.62          & 17.61          & 22.15          \\
                     & 60\%  & 19.93          & 17.49          & 22.28          & 23.62          & 23.63          & \textbf{15.77} \\
                     & 70\%  & 20.1           & 21.57          & \textbf{18.41} & 21.74          & 23.63          & 20.73          \\
                     & 80\%  & 18.86          & 23.09          & 17.36          & 23.63          & 20.66          & \textbf{15.48} \\
                     & 90\%  & \textbf{16.3}  & 23.58          & 23.63          & 22.12          & 17.37          & 16.32          \\
                     & 100\% & 17.64          & 23.63          & 23.63          & 23.63          & 23.44          & \textbf{15.75} \\
\hline
\multirow{10}{*}{10} & 10\%  & 8.49           & 48.53          & 12.23          & 7.24           & \textbf{6.06}  & 9.39           \\
                     & 20\%  & 8.65           & 28.61          & 3.58           & \textbf{3.04}  & \textbf{3.04}  & 6.2            \\
                     & 30\%  & 3.12           & \textbf{3.04}  & \textbf{3.04}  & \textbf{3.04}  & \textbf{3.04}  & 4.32           \\
                     & 40\%  & \textbf{2.72}  & 3.04           & 3.04           & 3.04           & 3.04           & 3.54           \\
                     & 50\%  & \textbf{3.04}  & \textbf{3.04}  & \textbf{3.04}  & \textbf{3.04}  & \textbf{3.04}  & 4.19           \\
                     & 60\%  & \textbf{3.04}  & 3.04           & \textbf{3.04}  & \textbf{3.04}  & \textbf{3.04}  & 3.63           \\
                     & 70\%  & \textbf{3.01}  & 3.04           & 3.04           & 3.04           & 3.04           & 3.5            \\
                     & 80\%  & \textbf{3.04}  & \textbf{3.04}  & \textbf{3.04}  & \textbf{3.04}  & \textbf{3.04}  & 3.58           \\
                     & 90\%  & \textbf{3.04}  & \textbf{3.04}  & \textbf{3.04}  & \textbf{3.04}  & \textbf{3.04}  & 3.34           \\
                     & 100\% & \textbf{3.03}  & 3.04           & 3.04           & 3.04           & 3.04           & 3.14           \\
\hline
\multirow{10}{*}{11} & 10\%  & \textbf{9.5}   & 9.51           & 9.51           & 9.84           & 9.51           & 9.51           \\
                     & 20\%  & 9.51           & 9.51           & 9.51           & 9.51           & 9.51           & 9.51           \\
                     & 30\%  & 9.51           & 9.51           & 9.51           & 9.51           & 9.51           & 9.51           \\
                     & 40\%  & 9.51           & 9.51           & 9.51           & 9.51           & 9.51           & 9.51           \\
                     & 50\%  & 9.51           & 9.51           & 9.51           & 9.51           & 9.51           & \textbf{9.5}   \\
                     & 60\%  & 9.49           & 9.51           & 9.51           & 9.51           & 9.51           & \textbf{9}     \\
                     & 70\%  & \textbf{9.51}  & \textbf{9.51}  & \textbf{9.51}  & \textbf{9.51}  & \textbf{9.51}  & 9.91           \\
                     & 80\%  & 9.51           & 9.51           & 9.51           & 9.51  & 9.51  & 9.51           \\
                     & 90\%  & 9.51           & 9.51           & 9.51           & 9.51           & 9.51           & \textbf{9.46}  \\
                     & 100\% & 9.51           & 9.51           & 9.51           & 9.51           & 9.51           & \textbf{9.37}  \\
\hline
\multirow{10}{*}{12} & 10\%  & 40.36          & 40.34          & 57.64          & 40.36          & \textbf{40.05} & 45.27          \\
                     & 20\%  & 40.36          & 40.36          & 40.36          & 40.36          & 39.63          & \textbf{39.56} \\
                     & 30\%  & \textbf{39.58} & 58.36          & 40.25          & 40.36          & 39.89          & 54.94          \\
                     & 40\%  & \textbf{40.27} & 40.36          & 40.36          & 40.36          & 40.36          & 50.98          \\
                     & 50\%  & \textbf{40.34} & 40.36          & 40.36          & 56.89          & 40.36          & 59.24          \\
                     & 60\%  & \textbf{40.34} & 40.36          & 40.36          & 57.81          & 40.36          & 55.41          \\
                     & 70\%  & \textbf{38.62} & 40.36          & 40.36          & 57.2           & 40.36          & 53.52          \\
                     & 80\%  & \textbf{39.87} & 40.36          & 40.36          & 40.36          & 40.36          & 59.65          \\
                     & 90\%  & 40.36          & 40.36          & 51.73          & 40.36          & \textbf{33.05} & 59.51          \\
                     & 100\% & \textbf{35.65} & 38.14          & 40.36          & 40.36          & 46.3           & 44.89   \\  
\hline
\label{table:results}
\end{longtable}

\begin{figure}[ht]
\centering
\includegraphics[width=0.7\textwidth]{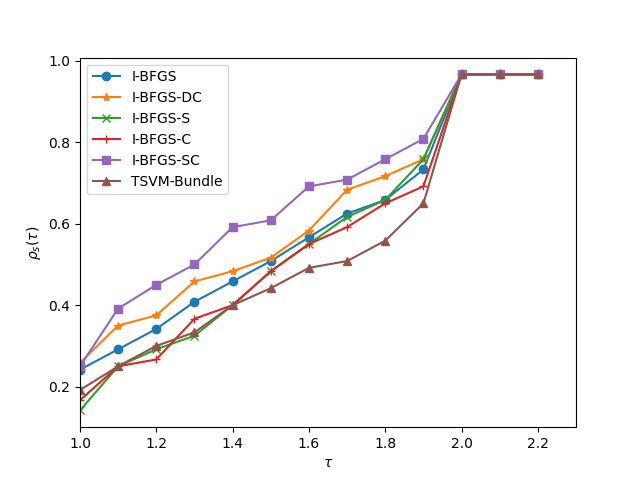}
\caption{Performance profile of algorithms.} \label{figure:performance_profile}
\end{figure}

\section{Conclusion}\label{conclusion}

In this paper, we have proposed and demonstrated the numerical performance of I-BFGS methods for the minimization of a sum of a finite number of functions that might be nonsmooth and/or nonconvex.  In particular, we focused on a problem that arises in semi-supervised machine learning.  Our results show that variants of I-BFGS that consider a difference-of-convex function formulation, smoothing, and/or (strongly) convex local approximations outperform a bundle method that has been designed for this specific problem formulation.

\section*{Disclosure statement}
The authors report there are no competing interests to declare.

\section*{Funding}
The first author is
supported by the Scientific and Technological Research Council of Türkiye (TUBITAK) with Postdoctoral Research Fellowship Program. 

\bibliographystyle{tfs}
\bibliography{ref}

\end{document}